\documentclass[11pt]{article}
\usepackage{amssymb,amsmath,amsthm,tikz}
\usepackage[normalem]{ulem}
\newtheorem{teorema}{Theorem}[section]
\newtheorem{lema}[teorema]{Lemma}

\newtheorem{ejemplo}[teorema]{Example}
\newtheorem{corolario}[teorema]{Corollary}
\newtheorem{teo}[teorema]{Theorem}
\newtheorem{prop}[teorema]{Proposition}
\newtheorem{remark}[teorema]{Remark}

\newcommand{\supp}{\hbox{\rm supp }}

\newcommand{\I}{\hbox{\rm I}}

\newcommand{\mult}{\hbox{\rm mult }}

\begin{document}

	\title{The topology of the generic polar curve and the Zariski invariant for  branches of genus one}
	\author{{\sc Evelia R.\ Garc\'{\i}a Barroso, Marcelo E. Hernandes} \\ {\sc and M. Fernando Hern\'andez Iglesias}\thanks{ The authors were partially supported by the Spanish grant PID2019-105896GB-I00 funded by MCIN/AEI/10.13039/501100011033. The second-named author was partially supported by CNPq-Brazil. The first-named and third-named authors were supported by  the Direcci\'on de Fomento de la Investigaci\'on at the PUCP grant DFI-2023-PI0983.}
	}
	
	\date{}
	
	\maketitle
	
	\begin{abstract}
		We study, for plane complex branches of genus one, the topological type of its generic polar curve, as a function of the semigroup of values and the Zariski  invariant of the branch. We improve some results given by Casas-Alvero in 2023, since we filter the topological type fixed for the branch by the possible values of Zariski invariants.
	\end{abstract}
	
	\noindent {\it Keywords: Plane branches, Zariski invariant, generic polar curve, topological type \\
		\noindent {2020 AMS Classification: 14H20, 32S10.}}\\
	\section{Introduction}
	Let $C_f:f(x,y)=0$ be a plane curve with $f(x,y)\in \mathbb C\{x,y\}$.
		The curve $C_f$ is called a branch when $f$ is irreducible. The multiplicity of the curve $C_f$, denoted by $n:=\mult(C_f)$, is by definition the order of $f$. The curve $C_f$ is singular if $n>1$, otherwise we say that $C_f$ is smooth.  After a change of coordinates, if necessary, we may assume that $x=0$ is not tangent to $C_f$ at $0$. By Newton Theorem there is $\alpha(x)=\sum_{i\geq n}a_ix^{i/n}\in \mathbb C\{x^{1/n}\}$,  such that $f(x,\alpha(x))=0$. The power series $\alpha(x)$ is called a Newton-Puiseux root of $f(x,y)$ and it determines a Puiseux parametrization of $C_f$ as follows:
	\[
	(x(t), y(t))=\left (t^n,\sum_{i\geq n}a_it^i\right ).
	\]
	
	Now, if $C_f$ is a branch then the set of all Newton-Puiseux roots of $f(x,y)$ is $\{\alpha_{\epsilon}(x)=\sum_{i\geq n}a_i\epsilon^i x^{i/n}\;:\;\epsilon \in \mathbb U_{n}\}$, where  $\mathbb U_{n}$ is the multiplicative group of the $n$th complex roots of unity.
	Let  $\alpha_{\epsilon}(x)=\sum_{i\geq n}a_i\epsilon ^i x^{i/n}$ be a Newton-Puiseux root of $C_f$. We associate with it two sequences $(e_{i})_{i}$ and $(\beta_{i})_{i}$ of natural numbers as follows: $e_{0}:=\beta_{0}:=n$; if $e_{k}> 1$ \label{page1}
	then $\beta_{k+1}:=\min\{i\,:\, a_i\neq 0,\, \gcd(e_k,i)<e_k\}$ and 
	$e_{k+1}:=\gcd(e_k, \beta_{k+1})$. 
	The sequence $(e_{i})_{i}$ is strictly decreasing and for some $g \in \mathbb N$ we have $e_g=1$. The number $g$ is the genus of $C_f$. The sequence $(\beta_{0},\beta_{1},\ldots,\beta_{g})$ is called the sequence of characteristic exponents of $\alpha(x)$ and coincides with the sequence of characteristic exponents of any other Newton-Puiseux root of $C_f$, hence this sequence is also called the sequence of characteristic exponents of the branch $C_f$.
	
	Two plane curves $C_{f}$ and $C_{h}$ are topologically equivalent (also called equisingular) if and only if there is a homeomorphism $\Psi:U\longrightarrow V$ where $U$ and $V$ are neighborhoods of the origin at $\mathbb C^2$ such that $f$ (respectively $h$) is convergent in $U$ (respectively in $V$) and $\Psi(C_f\cap U)=C_h\cap V$. It is well known (see \cite{zariski-equising}) that two branches $C_{f}$ and $C_{h}$ are topologically equivalent if and only if they have the same characteristic exponents and in such a case we will write $C_{f}\equiv C_{h}$. If the map $\Psi$ is an analytic isomorphism we say that $C_f$ and $C_h$ are analytic equivalent. We will denote by 	$K( 	n,\beta_1,\beta_2,\ldots,\beta_{g})$
	the set of equisingular branches of characteristic exponents $(n,\beta_1,\ldots,\beta_{g})$. If $C_f:f(x,y)=0$ is a branch in $K(n,\beta_1,\ldots,\beta_{g})$, then 
	we will put $f\in K(n,\beta_1,\ldots,\beta_{g})$.
	
	We associate to the branch  $C_{f}$ the set 
	\[
	\Gamma_f:=\{\I(f,h)\;:\; h\in \mathbb C\{x,y\},\;f \,\hbox{\rm does not divide }h\}
	\]
	\noindent where  $\I(f,h)=\dim_{\mathbb C}\mathbb C\{x,y\}/(f,h)$ is the intersection multiplicity of the curves $C_f$ and $C_h$ at the origin. 
	It is well-known that $\Gamma_f$ is a numerical semigroup called the semigroup of values of  $C_f$.  If $f\in K(n,\beta_1,\ldots, \beta_g)$ then the semigroup $\Gamma_f$ is finitely generated by $g+1$ natural numbers $v_0<v_1<\cdots <v_g$ and  there is a relationship between the sequence $\{v_i\}_{i=0}^g$ and the characteristic exponents $(n=\beta_0,\beta_1,\ldots,\beta_{g})$ of $C_f$ as follows:
	\begin{enumerate}
		\item $v_0=n=\beta_0$, $v_1=\beta_1$,
		\item $v_j=n_{j-1}v_{j-1}+\beta_j-\beta_{j-1}$ for $2\leq j\leq g$, where $n_{j-1}:=\frac{e_{j-2}}{e_{j-1}}$.
	\end{enumerate}
	We can then write
	$\Gamma_f=\langle v_0,\ldots ,v_g\rangle:=\mathbb N v_0+\mathbb Nv_1+\cdots +\mathbb N v_g$.
	
	Since the minimal system of generators of $\Gamma_f$ is equivalent to the set of characteristic exponents of the branch $C_f$, two branches $C_f$ and $C_h$ are topologically equivalent if and only if $\Gamma_f=\Gamma_h$.
	
	Let us now consider two reduced curves $C_f$ and $C_h$ where $f=f_1\cdots f_r$ (respectively $h=h_1\cdots h_s$) is  the decomposition of $f$ (respectively of $h$) into irreducible factors. The curves $C_f$ and $C_h$ are topologically equivalent if and only if $r=s$ and there is a bijection $\sigma:\{C_{f_i}\}_{i=1}^r \longrightarrow \{C_{h_i}\}_{i=1}^r$ such that $\sigma(C_{f_i})=C_{h_i}$ (at the cost of renumbering the branches of $C_h$), for all $i\in \{1,\ldots,r\}$ the branches $C_{f_i}$ and $C_{h_i}$ are topologically equivalent and for all $i,j$ with $1\leq i,j\leq r$, $\I(f_i,f_j)=\I(h_i,h_j)$. 
	
	The polar curve of $C_f$ with respect to a point $(a:b)$ of the complex projective line $\mathbb P^1(\mathbb C)$ is the curve $P_{(a:b)}(f): af_x+bf_y=0$. There exists an open
	Zariski set $U$ of $\mathbb P^1(\mathbb C)$ such that $\{P_{(a:b)}(f)\;:\; (a : b) \in U\}$ is a
	family of topologically equivalent plane curves. Any element of this set is called generic polar
	curve of  $C_f$ and we will denote  it by $P(f)$. It is well known that the topological type of $P(f)$ depends on the analytical type of $C_f$ (see \cite[Exemple 3]{Pham}). 
	
	In \cite{zariski-torsion}, Zariski introduces an analytical invariant for $C_f$ called the Zariski invariant (see Section \ref{Zariski combinatory}). The aim of this work is to study the behaviour of the topological type of the generic polar curve with respect to the Zariski invariants for branches of genus one.
	
	Casas, in \cite{Casas}, presents results considering some analytic invariants of $C_f$ that influence on the topological type of the generic polar curve $P(f)$ for a plane curve $C_f$ in $K(n,m)$.
	In Section 2 we recall the concept of a non-degenerate curve. In Section 3 we study the so-called Zariski invariant. In Section \ref{Polars} we study the topological type of the generic polar curves of plane branches in $\mathcal{L}(n,m,\lambda)$ that is branches in 
	$K(n,m)$ with a fixed Zariski invariant $\lambda$. This provides a finner approach to understanding the behaviour of the topological type of $P(f)$ in terms of the Zariski invariant of $C_f$.
	
	\section{Topological type for a non-degenerate curve}\label{nondegerated}
	
	Let us remember the notion of the Newton polygon of a curve. Let $S\subseteq \mathbb N^2$. The Newton diagram of $S$, denoted by $\mathcal{ND}(S)$, is by definition the convex hull of $S+(\mathbb R_{\geq 0})^2$, where $+$ denotes the Minkowski sum. The Newton polygon $\mathcal {NP}(S)$ of $S$ is the compact polygonal boundary of $\mathcal {ND}(S)$. 
	
	The support of any power series $f(x,y)=\sum_{i,j} a_{i,j}x^iy^j\in \mathbb C\{x,y\}\setminus\{0\}$ is $\supp(f):=\{(i,j)\in \mathbb N^2\;:\;a_{i,j}\neq 0\}$. The Newton diagram of $f$, denoted by $\mathcal{ND}(f)$, is $\mathcal{ND}(\supp(f))$. The Newton
	polygon $\mathcal {NP}(f)$ of $f$ is the Newton polygon of  $\supp(f)$. The
	Newton diagram of $f$ depends on coordinates but $\mathcal {ND}(uf)=\mathcal {ND}(f)$ for any unit $u\in \mathbb C\{x,y\}$. Hence the Newton diagram of a curve is by definition the Newton diagram of any of its equations.
	
	Let $L$ be a compact edge of $\mathcal {NP}(f)$. Denote by $\vert L\vert_1$ (respectively $\vert L\vert_2$) the  length of the projection of $L$ over the horizontal (respectively vertical) axis. The inclination of $L$ is $i_L:=\frac{\vert L\vert_1}{\vert L\vert_2}$.
	By \cite[Lemme 8.4.2]{Chenciner}, if $L$ is a compact edge of $\mathcal {NP}(f)$  then the curve $C_f$ has $\vert L\vert_2$ Newton–Puiseux roots of order $i_L$.
	
	On the other hand, if $f(x,y)=\sum_{i,j}a_{i,j}x^iy^j\in \mathbb C\{x,y\}\setminus\{0\}$ without multiple irreducible factors, then we associate with any compact edge $L$ of $\mathcal {NP}(f)$ the polynomial $f_L(x,y)=\sum_{(i,j)\in L \cap \supp(f)}a_{i,j}x^iy^j\in \mathbb C[x,y]$. We say that $f$ is Newton non-degenerate with respect to the coordinates $(x,y)$ if for any compact edge $L$ of $\mathcal {NP}(f)$, the polynomial $f_L(x,y)$ has no critical points outside the axes $x=0$ and $y=0$, which is equivalent to the non-zero roots of the $y$-polynomial $f_L(1,y)$ are simple. 
	According to Oka, we have
	
	\begin{prop}(\cite[Proposition 4.7]{Oka})\label{Oka}
		Let $C_f$ be a plane curve without multiple irreducible components and consider a coordinate system $(x,y)$ such that $x=0$ is not tangent to $C_f$. If $C_f$
		is non-degenerate with respect to $(x,y)$, then associated with any compact edge $L$ of $\mathcal {NP}(f)$  there are $d_L =\gcd(\vert L \vert_1, \vert L \vert_2)$ branches 
		$\{C_{f^{(L)}_i}\}_{i=1}^{d_L}$ of 
		$C_f$ with $\mult(C_{f^{(L)}_i})=\frac{\vert L \vert_2}{d_L}$  
		for any $i$ and if $C_{f^{(L)}_i}$ is singular then its semigroup of values is  $\Gamma_{f^{(L)}_i}=\left \langle \frac{\vert L \vert_2}{d_L}, \frac{\vert L \vert_1}{d_L}\right \rangle$. Moreover for any two compact edges $L,L'$ (not necessarily different) of $\mathcal {NP}(f)$ and $f_i^{(L)}\neq f_j^{(L')}$ we get $\I(f^{(L)}_i, f^{(L')}_j)=\min\left \{ \frac{\vert L\vert _1 \vert L'\vert _2}{d_Ld_{L'}}, \frac{\vert L\vert _2 \vert L'\vert _1}{d_Ld_{L'}}\right \}$.
	\end{prop}
	
	\section{The Zariski invariant}\label{Zariski combinatory}
	
	Let $f(x,y)\in \mathbb C\{x,y\}$ irreducible and consider the plane branch $C:f(x,y)=0$. Suppose that  $f\in K(n,m)$. In particular, the semigroup associated to $C_f$ is $\Gamma_f=\langle n,m\rangle$.
	
	Up to an analytic change of coordinates we can suppose that $f\in\mathbb{C}\{x\}[y]$ is a Weierstrass polynomial with $\deg_y(f)=\mult(C_f)=n$.
	
	Zariski proved in \cite[pages 785-786]{zariski-torsion} and \cite[Chapitre III.1]{Zariski-libro} that if $f\in K(n,m)$ then there exists a change of coordinates such that $C_f$ is analytically equivalent to a plane branch with Puiseux parametrization $(t^n,t^m)$ or 
	\[\left(t^n, t^m+t^{\lambda_f}+ \sum_{i>\lambda_f}a_it^i\right ),\]
	where $\lambda_f+n\not\in\Gamma_f$ is an analytical invariant called the Zariski invariant of $C_f$.  In particular $\lambda_f>m$. If $C_f$ is analytically equivalent to $y^{n}-x^{m}=0$ then we put $\lambda_f=\infty$. 
	
	Any branch $C_f$ such that $n=2$ or $f\in K(3,m)$ with $m\in\{4,5\}$ has Zariski invariant equal to $\lambda_f=\infty$ (see \cite[Remark 1.2.10]{handbook}). So, in what follows we consider branches $C_f$ with multiplicity $n> 3$ or $f\in K(3,m)$ with $m\geq 7$.
	
	According to the definition we get that the set of possible finite Zariski invariants of branches in $K(n,m)$ is
	\[
	\mathcal{Z}(n,m):=\{\lambda\in\mathbb{N}\;:\; \lambda>m\ \mbox{and}\ \lambda+n\not\in\langle n,m\rangle\}.
	\]
	
	\begin{lema}\label{landas}
		Any  $\lambda\in \mathcal{Z}(n,m)$ can be uniquely expressed by 
		\[\lambda=i m-jn\ \ \mbox{for}\ \ 2\leq\ i\leq n-1\ , \ 2\leq\ j < \frac{(i-1)m}{n}, \, \text{and} \,\, (i, j) \in \mathbb{N}^2. \]
	\end{lema}
	\begin{proof}
		Recall that any $z\in\mathbb{Z}$ can be uniquely represented as	$$ z=s_0n+s_1m\ \ \mbox{with}\ \  0\leq s_1< n\ \ \mbox{and}\ \ s_0\in\mathbb{Z}.$$ Moreover, $z=s_0n+s_1m\in\langle n,m\rangle$ if and only if $s_0\geq 0$.
		
		Since $\lambda+n\not\in \langle n,m \rangle$ we get $\lambda\not\in \langle n,m \rangle$. In this way, there are $s_0,s_1\in \mathbb Z$ such that $\lambda=s_0n+s_1m$ with $s_0< 0$ and $0\leq s_1\leq n-1$. As, $\lambda=s_0n+s_1m>m$ then $s_1\geq 2$ and $\frac{(1-s_1)m}{n}<s_0$. Moreover $s_0\leq -2$ (otherwise $\lambda +n$ belongs to  $\langle n,m \rangle$).
	\end{proof}

	Let $T\subseteq\mathbb{R}^2$ be the triangle determined by the lines
	\[y=n-1,\ \ x=m-1\ \ \mbox{and}\ \ my+nx=nm,\]
	that is, the triangle with vertices $(m-1,n-1), (m-1,\frac{n}{m})$ and $(\frac{m}{n},n-1)$ (see the left side of Figure \ref{fig:triangle2en1}). Denote by
	$\overset{o}{T}$ the set of lattice points in the interior of $T$.

	\begin{figure} 
		\begin{center}
			\begin{tikzpicture}[x=0.45cm,y=0.45cm] 
				
				\begin{scope}[shift={(0,0)}]
					\fill[fill=yellow!40!white] (12/5,4) --(11,4)-- (11,5/12) -- (12/5,4) --cycle;
					\node[draw,circle,inner sep=1.5pt,fill, color=black] at (0,5){};
					\node [left] at (0,5) {$n$};
					\node[draw,circle,inner sep=1.5pt,fill, color=black] at (0,4){};
					\node [left] at (0,4) {$n-1$};
					\node[draw,circle,inner sep=1.5pt,fill, color=black] at (12,0){};
					\node [below] at (12.1,-0.1) {$m$};
					\node[draw,circle,inner sep=1.5pt,fill, color=black] at (11,0){};
					\node [below] at (10.5,0) {$m-1$};
					\node[draw,circle,inner sep=1.5pt,fill, color=black] at (0,5/12){};
					\node [left] at (0,5/12) {$n/m$};
					\node[draw,circle,inner sep=1.5pt,fill, color=black] at (12/5,0){};
					\node [below] at (12/5,0) {$m/n$};
					\node [below] at (8,3.5) {$\overset{o}{T}$};
					\draw[->] (0,0) -- (13,0) node[right,below] {$i$};
					\draw[->] (0,0) -- (0,6) node[above,left] {$j$};
					\draw[-, line width=0.5mm] (0,5) -- (12,0);
					\draw[-, line width=0.5mm] (0,4) -- (12,4);
					\draw[-, line width=0.5mm] (11,0) -- (11,5);
					\draw[-][dashed] (13,5/12) -- (0,5/12);
					\draw[-][dashed] (12/5,0) -- (12/5,5);
				\end{scope}
				
				\begin{scope}[shift={(15,0)}]
					\tikzstyle{every node}=[font=\small]
					\fill[fill=yellow!40!white] (12/5,4) --(11,4)-- (11,5/12) -- (12/5,4) --cycle;
					\foreach \x in {0,...,12}{
						\foreach \y in {0,...,5}{
							\node[draw,circle,inner sep=1pt,fill, color=gray!40] at (1*\x,1*\y) {}; }
					}
					
					\foreach \x in {5,...,10}{
						\foreach \y in {3}{
							\node[draw,circle,inner sep=1pt,fill, color=black] at (1*\x,1*\y) {}; }
					}
					
					\foreach \x in {8,...,10}{
						\foreach \y in {2}{
							\node[draw,circle,inner sep=1pt,fill, color=black] at (1*\x,1*\y) {}; }
					}
					\foreach \y in {1,2,...,5} \draw(0,\y)node[left]{\y};
					\foreach \x in {0,1,...,12} \draw(\x,0)node[below]{\x};
					\node[draw,circle,inner sep=1pt,fill, color=black] at (10,1){};
					\node [above] at (5,3) {13};
					\node [above] at (10,1) {14};
					\node [above] at (8,2) {16};
					\node [above] at (6,3) {18};
					\node [above] at (9,2) {21};
					\node [above] at (7,3) {23};
					\node [above] at (10,2) {26};
					\node [above] at (8,3) {28};
					\node [above] at (9,3) {33};
					\node [above] at (10,3) {38};
					
					\draw[->] (0,0) -- (13,0) node[right,below] {$i$};
					\draw[->] (0,0) -- (0,6) node[above,left] {$j$};
					\draw[-, line width=0.5mm] (0,5) -- (12,0);
					\draw[-, line width=0.5mm] (0,4) -- (12,4);
					\draw[-, line width=0.5mm] (11,0) -- (11,5);
					
				\end{scope}
				
			\end{tikzpicture}
		\end{center}
		\caption{The triangle $T$ and the elements of  $\mathcal{Z}(5,12)$.}  
		\label{fig:triangle2en1}
	\end{figure}
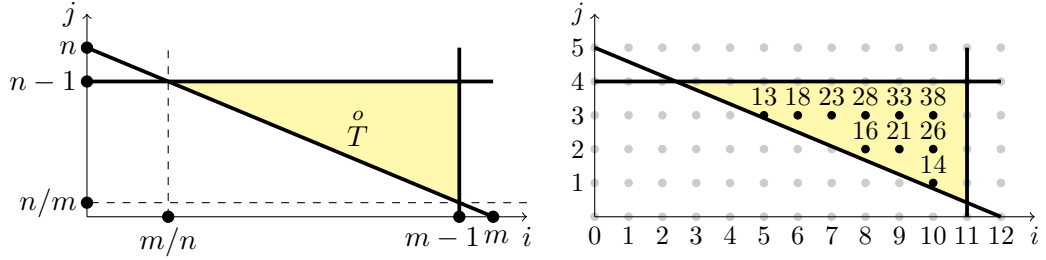

	The following proposition gives the relation between $\mathcal Z(n,m)$ and  $\overset{o}{T}$.
	
	\begin{prop}\label{biyeccion}
		The map 
		\begin{equation}\label{Phi}
			\begin{array}{rll}
				\Phi: \mathcal{Z}(n,m) &\longrightarrow &\overset{o}{T}\\
				\lambda=\alpha m - \beta n& \longrightarrow & \Phi(\alpha m - \beta n)=(m-\beta,\alpha-1),
			\end{array}
		\end{equation}
		\noindent where $\ 2\leq \alpha \leq n-1$ and $2\leq \beta < \frac{(\alpha -1)m}{n}$ is a bijection between the set $\mathcal Z(n,m)$ and the lattice points in the interior of $T$.
	\end{prop}
	
	\begin{proof}
		By Lemma \ref{landas} we have that $\Phi$ is an injective map.
		
		On the other hand,  if $(i,j)\in \overset{o}{T}$ then
		$in+jm=(n-1)m+\lambda$ for some $\lambda>m$, $2 \leq i \leq m-2, \ \ 1 \leq j \leq n-2$, so $\lambda=m(j+1)-n(m-i)$.
		Therefore if $m-i=\beta \ \text{and} \ j+1=\alpha$, then $ 2 \leq \alpha \leq n-1$ and $2 \leq \beta < \frac{(\alpha-1)m}{n}$, consequently, $\lambda\in\mathcal{Z}(n,m)$.
		Hence, $\Phi$ is a bijection.
	\end{proof}
	
	In the right side of Figure \ref{fig:triangle2en1} 
	we illustrate the image of bijection $\Phi$ introduced in Proposition \ref{biyeccion} for the case
	$\mathcal Z (5,12)$.
	
	Let us consider in $\overset{o}{T}$ the following $(n,m)$-weighted ordering $\prec$\ : 
	
	If $(x_0,y_0)$, $(x_1,y_1)\in \overset{o}{T}$ then we put
	
	\[(x_0,y_0)\prec (x_1,y_1)\ \hbox{\rm if and only if }\ x_0n+y_0m<x_1n+y_1m.
	\]
	
	In particular, the bijection $\Phi$ preserves orders, that is 
	
	\[\alpha_0m-\beta_0n<\alpha_1m-\beta_1n\ \mbox{if and only if}\ \Phi(\alpha_0m-\beta_0n)\prec \Phi (\alpha_1m-\beta_1n).
	\]
	
	The triangle $T$ was spotted by Zariski in \cite[page 107]{Zariski-libro}.
	
	Let $\lambda \in \mathcal Z(n,m)$ and $\Phi(\lambda)=(p,q)$. We define
	\begin{eqnarray}\label{singularidad}
		{\mathcal I}_{\lambda}: & =&\{ (i,j)\in \overset{o}{T} \;:\;
		(i,j)\succ (p,q)\}\\ 
		& =&\{ (i,j)\in \mathbb{N}^2 \;:\;
		in+jm>pn+qm,\; 0\leq i\leq m-2, \;0\leq j\leq n-2\}\nonumber.
	\end{eqnarray}
	
	In Figure \ref{fig:NPf} we illustrate the set ${\mathcal I}_{\lambda}$. 
	
	\begin{figure}[h!] 
		\begin{center}
			\begin{tikzpicture}[x=0.5cm,y=0.5cm] 
				\tikzstyle{every node}=[font=\small]
				\foreach \x in {0,...,12}{
					\foreach \y in {0,...,5}{
						\node[draw,circle,inner sep=1pt,fill, color=gray!40] at (1*\x,1*\y) {}; }
				}
				\fill[fill=blue!30!white] (3.6,4) -- (11,4) -- (11,11/12) --(3.6,4) --cycle;
				\foreach \x in {7,8,9,10}{
					\foreach \y in {3}{
						\node[draw,circle,inner sep=1pt,fill, color=black] at (1*\x,1*\y) {}; }
				}
				
				\foreach \x in {9,10}{
					\foreach \y in {2}{
						\node[draw,circle,inner sep=1pt,fill, color=black] at (1*\x,1*\y) {}; }
				}	
				\node [above] at (6.2,3) {$(p,q)$};
				\node [left] at (0,5) {$n$};
				\node [left] at (-0.1,4) {$n-1$};
				\node [below] at (12.1,-0.1) {$m$};
				\node [below] at (10.75,0) {$m-1$};
				\draw[->] (0,0) -- (14,0) node[right,below] {$i$};
				\draw[->] (0,0) -- (0,6) node[above,left] {$j$};
				\draw[-, line width=0.5mm] (0,5) -- (12,0);
				\draw[-, line width=0.5mm] (0,4) -- (12,4);
				\draw[-, line width=0.5mm] (11,0) -- (11,5);
				\draw[-][dashed]  (0,5.5) -- (13.2,0);
				\node[draw,circle,inner sep=1.5pt,fill, color=black] at (0,5){};
				\node[draw,circle,inner sep=1.5pt,fill, color=black] at (0,4){};
				\node[draw,circle,inner sep=1.5pt,fill, color=black] at (11,0){};
				\node[draw,circle,inner sep=1.5pt,fill, color=black] at (12,0){};
				\node[draw,circle,inner sep=1.5pt, color=black] at (6,3){};
			\end{tikzpicture}
		\end{center}
		\caption{Points $(i,j)\in \overset{o}{T}$ with $in+jm> pn+qm$.}  
		\label{fig:NPf}	
		
	\end{figure}
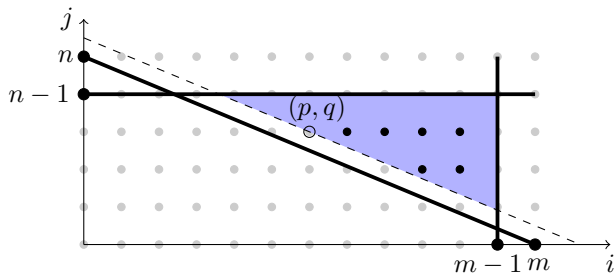
	
	In \cite[Lemma 1.4]{Peraire} Peraire stablished a bijection between $\overset{o}{T}$ and the set $\{s:=z-m\,:\; z\in \mathcal Z(n,m)\}$.
	Moreover, in \cite[Theorem 1.5]{Peraire} Peraire proved that if $f\in K(n,m)$ with Zariski invariant $\lambda_f=m+s$ with $s>0$,  then there are  analytic coordinates $(x,y)$ such that 
	\begin{equation}\label{eqf}
		f(x,y)=y^n-x^m+x^py^q+\sum_{(i,j)\in \mathcal J}
		a_{i,j}x^iy^j,\ \ \mbox{for some}\ \ a_{i,j}\in\mathbb{C}
	\end{equation}
	where $\mathcal J:=\{ (i,j)\in \mathbb{N}^2 \;:\;
	in+jm>mn+s,\;0\leq i\leq m-2, \;0\leq j\leq n-2\}$ and $np+qm=mn+s$. 
	
	\noindent Notice that $np+mq=mn+\lambda_f-m$, that is, $\lambda_f=(q+1)m-(m-p)n$. Since $\lambda_f\in\mathcal{Z}(n,m)$, by (\ref{Phi}), we get $\Phi (\lambda_f)=(p,q)\in \overset{o}{T}$. Moreover, the sets $\mathcal J$ and $\mathcal I_{\lambda_f}$ are equal.
	In this way, we can rewrite (\ref{eqf}) as
	\begin{equation}\label{equationf}
		f(x,y)=y^n-x^m+x^py^q+\sum_{
			(i,j)\in \mathcal{I}_{\lambda}} a_{i,j}x^iy^j.
	\end{equation}

	\begin{remark}\label{p+q}
		Since $f$ is irreducible with order $n$, in (\ref{equationf}) we get: 
		$$n<p+q<m \,\,\,\text{if} \,\,\, \lambda_f<2m-n.$$\\
			 Indeed, we have $\Phi(\lambda_f)=(p,q)$ and $np+mq=(n-1)m+\lambda_f$, so if $\lambda_f<2m-n$ then we get 			$np+mq<(n+1)m-n$ and
			$p+q<\frac{nm+(m-n)(p+1)}{m}$. 
The inequality follows  because $p\leq m-2$ and $n<m$.	
		
	\end{remark}
	
	\section{Topological type of generic polar curves of branches in $\mathcal L(n,m,\lambda)$}
	\label{Polars}
	
	Let $n$ and $m$ be two coprime integers such that $2<n<m$. Fix $\lambda\in\mathcal{Z}(n,m)\cup\{\infty\}$. Denote by $\mathcal L(n,m,\lambda)$ the set of  branches in $K(n,m)$ with Zariski invariant equals $\lambda$.
	
	Any branch in $\mathcal{L}(n,m,\infty)$ is analytically equivalent to the branch
	$y^n-x^m=0$. On the other hand, if $\lambda\in \mathcal{Z}(n,m)$ then any branch $C_f\in\mathcal{L}(n,m,\lambda)$ 
	admits, up to change of analytic coordinates, an equation $f$	as (\ref{equationf}).
	
	Since we are interesting to study the topological type of generic polar curves of 
	$C\in\mathcal L(n,m,\lambda)$ and, as noted in the introduction, it depends on the analytical type of $C$, from now on, for abuse of language, we will write $f\in\mathcal L(n,m,\lambda)$ and when we do so we will 
	assume that $f$ is	as in 
	(\ref{equationf}).
	
	By Lemma \ref{landas} and Proposition \ref{biyeccion}, if $\lambda\in \mathcal Z(n,m)$ then there are $\alpha, \beta \in \mathbb N$ with $2\leq\alpha <n$ and $2\leq\beta<\frac{(\alpha-1)m}{n}$ such that
		$\lambda=\alpha m - \beta n $ and $ \Phi(\lambda)=(m-\beta,\alpha-1)=:(p,q)$,
		where  $\Phi$ was defined in \eqref{Phi}.
		
		By definition, the point $(p,q) \in \overset{o}{T}$
		corresponding to the invariant  $\lambda$ verifies the condition $np+qm=(n-1)m+\lambda>mn$.
		
		We will study, via its Newton diagram, the topological type of the generic polar curve of a branch in $\mathcal L(n,m,\lambda)$.
		
		\subsection{The set $T_\lambda$}\label{Tlambda}
		
		Given $\lambda\in\mathcal{Z}(n,m)$ and $f\in\mathcal{L}(n,m,\lambda)$  as in (\ref{equationf}). Let  us
		consider the generic polar curve $P(f): a f_x +b f_y=0$ of $f$. If 
		
		\[	\Theta_f:=\{(p-1,q)\}\cup\{(i-1,j), (i,j-1)\in\mathbb{N}^2:\ (i,j)\in \mathcal{I}_{\lambda}\cap\supp (f)\}
		\]
  \noindent and
	 \[E_{\lambda}:=\{(0,n-1), (m-1,0), (p,q-1)\}\] 
			then the Newton polygon of the generic polar curve $P(f)$ equals the Newton polygon of $E_{\lambda}\cup \Theta_f$.
			
		Put $N:=(0,n-1), M:=(m-1,0)$, $A:=(p,q)$ and $B:=(p,q-1)$.
		
		If $F, G\in\mathbb{R}^2$, then we denote by $l_{F,G}$ the real line passing by $F$ and $G$. Notice that $l_{N,M}: (n-1)x+(m-1)y=(n-1)(m-1)$ and $l_{N,B}:(n-q)x+py=(n-1)p$ whose slope is $\frac{q-n}{p}>-1$ because $n<p+q$ (see Remark \ref{p+q}).
		
		We say that a point $(\alpha_1,\alpha_2)\in \mathbb R^2$ is above (respectively below) the line $\ell: ax+by=c$ if $a\alpha_1+b\alpha_2>c$ (respectively $a\alpha_1+b\alpha_2<c$). 
		
		Given $H=(h_1,h_2)\in\mathbb{R}^2$ we denote by $l_{\sigma,H}:nx+my=h_1n+h_2m$ the line with slope $\sigma:=-\frac{n}{m}$ and passing through the point $H$. By (\ref{singularidad}), it follows that any $(i,j)\in \mathcal{I}_{\lambda}$ is above $l_{\sigma,A}:nx+my=pn+qm$.
		
		\begin{lema}\label{lado-1}
			Let $f\in \mathcal L(n,m,\lambda)$ with $\Phi(\lambda)=(p,q)$. Any point $(\alpha_1,\alpha_2)\in \Theta_f$ with $q-1<\alpha_2<n-1$ is above the line	
			$l_{N,B}$.
		\end{lema}
		\begin{proof}
			Since $\frac{q-n}{p}>-1$ (see Remark \ref{p+q}), we deduce from the equation of $l_{N,B}$ that $(p-1,q)$ is above this line. On the other hand 
			any point $(i,j)\in \mathcal I_{\lambda}$ is above $l_{\sigma,A}$, then $(i-1,j)$ and $(i,j-1)$ are above $l_{\sigma,B}$ for every $(i,j)\in\mathcal I_{\lambda}$. Since the slope $\frac{q-n}{p}$ of $l_{N,B}$ is greater than the slope of $l_{\sigma,B}$, we have that $(i-1,j)$ and $(i,j-1)$ are above $l_{N,B}$ for every $(i,j)\in\mathcal I_{\lambda}$ with $q-1<j<n-1$.
		\end{proof}
		
		Notice that by Lemma \ref{lado-1} the Newton polygon of the generic polar curve $P(f)$ for any $f\in\mathcal{L}(n,m,\lambda)$ is the Newton polygon of the set 
		\begin{equation}\label{eq:points}
			E_{\lambda}\cup \{(\alpha_1,\alpha_2)\in \Theta_f\;:\;0\leq\alpha_2\leq  q-2\}.
		\end{equation}
		Observe that the points on \eqref{eq:points}  arise from the following points of $\supp (f)$:  
		\[\{(0,n), (m,0), (p,q)\}\cup \{(i,j)\in\mathcal{I}_{\lambda}\cap\supp (f):\ 0\leq j\leq q-1.\]
				
		In order to study the topological type of the generic polar of an element $f\in\mathcal{L}(n,m,\lambda)$ we will use the sets $E_{\lambda}$ and $T_{\lambda}$  defined in the sequel.
		
		Put $Q:=(m-1,1)$. Denote by $T_{\lambda}$ the set of points $(\alpha_1,\alpha_2)\in\mathbb{N}^{2}$ with $p<\alpha_1<m-1$, $0< \alpha_2<q$ such that $(\alpha_1,\alpha_2)$ is above $l_{\sigma, A}$ and below or on the line  $l_{A,Q}:(m-1-p)(y-q)+(q-1)(x-p)=0$. From the equations of the lines $l_{\sigma, A}$ and $l_{A,Q}$ we get
		\begin{equation}\label{regiont1}
			T_{\lambda}=\left \{\begin{array}{ll}(\alpha_1,\alpha_2)\in \mathbb{N}^2: & p<\alpha_1<m-1,\; \; 0<\alpha_2<q\vspace{0.2cm}\\ & \mbox{and}\ \  \frac{q-1}{m-1-p}\leq\frac{q-\alpha_2}{\alpha_1-p}<\frac{n}{m}\end{array}\right \}.\end{equation}
		
		Figure \ref{fig:T1} illustrates $T_{\lambda}$ for $\lambda=13\in \mathcal Z(5,12)$, in this case $T_{\lambda}=\{(10,1),(8,2)\}$.
		
		\begin{figure}[h!] 
			\begin{center}
				\begin{tikzpicture}[x=0.8cm,y=0.8cm] 
					\tikzstyle{every node}=[font=\small]
					\fill[fill=yellow!40!white] (12/5,4) --(11,4)-- (11,5/12) -- (12/5,4) --cycle;
					\fill[fill=green] (5,3) --(11,1)-- (49/5,1) -- (5,3) --cycle;
					
					\foreach \x in {0,...,12}{
						\foreach \y in {0,...,5}{
							\node[draw,circle,inner sep=1pt,fill, color=gray!40] at (1*\x,1*\y) {}; }
					}
					\draw[-][dashed, line width=0.4mm] (-1,33/6) -- (13,-1/3);
					\node [above] at (5.8,3) {\tiny{$A=(p,q)$}};
					\node [below] at (3.8,2) {\tiny{$(p,q-1)=B$}};
					\node [left] at (0,5) {$n$};
					\node [below, left] at (0,3.9) {\tiny{$(0,n-1)=N$}};
					\node [below] at (12,0) {$m$};
					\node [below,left] at (11.1,-0.3) {\tiny{$(m-1,0)=M$}};
					\node [right] at (11,1.2) {\tiny{$Q=(m-1,1)$}};
					\node [right] at (11.1,0.6) {$R$};
					\node [below] at (13,-1/3) {$l_{\sigma ,A}$};
					\node [above] at (9,3) {$\overset{o}{T}$};
					
					\draw[->] (0,0) -- (13.5,0) node[right,below] {$i$};
					\draw[->] (0,0) -- (0,6) node[above,left] {$j$};
					\draw[-, line width=0.5mm] (0,5) -- (12,0);
					\draw[-, line width=0.5mm] (0,4) -- (12,4);
					\draw[-, line width=0.5mm] (11,0) -- (11,5);
					\draw[-, line width=0.5mm] (5,2) -- (11,0);
					\draw[-, line width=0.5mm] (0,4) -- (5,2);
					\draw[-, line width=0.5mm , color=green!80!black] (5,3) -- (11,1); 
					\draw[-, line width=0.5mm , color=green!80!black] (49/5,1) -- (11,1); 
					
					\draw[->][thick, color=black] (8.2,2) .. controls (10,1.8) .. (11.6,2.3);
					\draw[->][thick, color=black](10,1.2) .. controls (10,1.8) .. (11.6,2.3);
					\node [above] at (12,2) {$T_{\lambda}$};
					
					\node[draw,circle,inner sep=1.5pt,fill, color=black] at (11,1/2){};
					\node[draw,circle,inner sep=1.5pt,fill, color=black] at (5,2){};
					\node[draw,circle,inner sep=1.5pt,fill, color=black] at (11,0){};
					\node[draw,circle,inner sep=1.5pt,fill, color=black] at (12,0){};
					\node[draw,circle,inner sep=1.5pt,fill, color=black] at (0,4){};
					\node[draw,circle,inner sep=1.5pt,fill, color=black] at (0,5){};
					\node[draw,circle,inner sep=1.5pt,fill, color=black] at (10,1){}; 
					\node[draw,circle,inner sep=1.5pt,fill, color=black] at (8,2){}; 
					\node[draw,circle,inner sep=1.5pt, color=black] at (5,3){};
					\node[draw,circle,inner sep=1.5pt, color=black] at (11,1){};
				\end{tikzpicture}
			\end{center}
			\caption{Points in $T_{\lambda}$ for $\lambda=13\in\mathcal{Z}(5,12)$.} 
			\label{fig:T1}
		\end{figure}
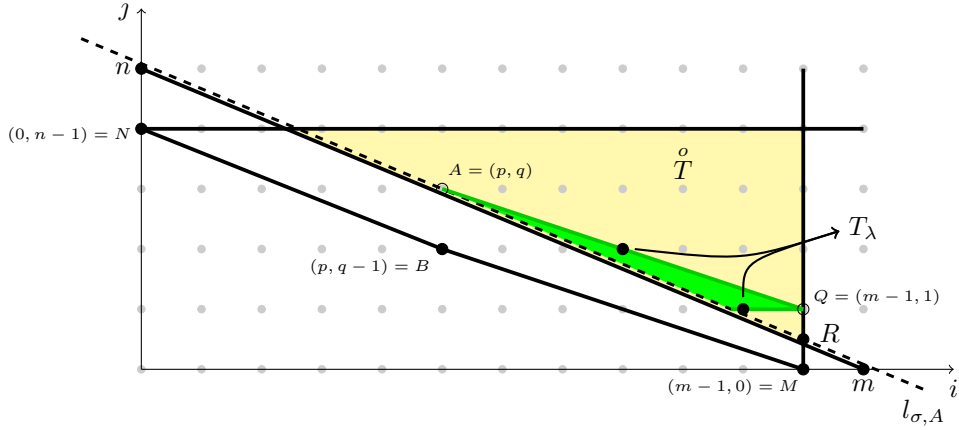
		
		Now we highlight some cases where
		$T_{\lambda}=\emptyset$.
		
		\begin{lema}\label{T=0}
			Let $\lambda\in \mathcal Z(n,m)$ with $\Phi(\lambda)=(p,q)$. We get $T_{\lambda}=\emptyset$ in the following cases:
			\begin{enumerate}
				\item $\lambda>2m-n$.
				\item $\Phi(\lambda)=(p,1)$, that is $q=1$.
				\item For any $\lambda\in K(3,m)$ with $m\geq 7$.  
				\item For any $\lambda\in K(n,n+1)$.
			\end{enumerate}
		\end{lema}
		
		\begin{proof}  Since the intersection point of $l_{\sigma,A}$ with $x=m-1$ is $R=\left ( m-1,\frac{\lambda+n-m}{m}\right )$ it follows that $\lambda>2m-n$ if and only if $R=(m-1,\delta)$ with $\delta>1$. But the intersection point of $l_{A,Q}$ with $x=m-1$ is $Q=(m-1,1)$. In particular, if $\lambda>2m-n$ we get $T_{\lambda}=\emptyset$ and the item 1) follows.   
			
			The item 2) follows directly from the definition of $T_{\lambda}$ (see \eqref{regiont1}).
			
			For the third statement, recall that any $\lambda\in\mathcal{Z}(3,m)$ with $m\geq 7$ is expressed, according to Lemma \ref{landas}, as $\lambda=2m-3j$ for some $2\leq j<\frac{m}{3}$, so $\Phi(\lambda)=(m-j,1)$ and the item 2) gives the result.
			
			For the item 4), remark that $n+2\not\in \mathcal Z(n,n+1)$ since $n+2+n=2(n+1)\in \langle n,n+1\rangle$, so for any $\lambda\in\mathcal{Z}(n,n+1)$  we get
			$\lambda> n+2=2(n+1)-n$, that is we have the condition of item 1) and the lemma follows.
		\end{proof}
		
		In the following subsection we will show how the set $T_{\lambda}$ can be used to determine the Newton polygon of the generic polar curve of $f\in \mathcal{L}(n,m,\lambda)$. 
				
		\subsection{Topological type of $P(f)$ for $f\in\mathcal{L}(n,m,\lambda)$}
		
		We recall that if $C_f$ is such that $\lambda_f=\infty$ then, according to \cite[pages 785-786]{zariski-torsion}, the plane branch is analytically equivalent to $y^n-x^m=0$ with $\gcd (n,m)=1$ and in this case, the topological type of the generic polar curve $P(f)$ is the same of the topological type of $y^{n-1}-x^{m-1}=0$. Indeed  by Proposition \ref{Oka}  that is, if $d:=\gcd(n-1,m-1)$ then $P(f)$ have $d$ irreducible equisingular components $\{P_i\}_i^{d}$ such that $\Gamma_{P_i}=\left \langle \frac{n-1}{d},\frac{m-1}{d} \right \rangle$ for any $i\in \{1,\ldots,d\}$ and $\I(P_i,P_j)=\frac{(n-1)(m-1)}{d^2}$ for any $i\neq j$.
		
		The following theorem  highlights the relevance of the region $T_{\lambda}$ to describe the Newton polygon of $P(f)$ of a plane branch $f\in\mathcal{L}(n,m,\lambda)$.
		
		\begin{teo}\label{prop-triang1}
			For any $f=y^n-x^m+x^{p}y^{q}+\sum_{(i,j)\in \mathcal I_{\lambda} }a_{i,j}x^iy^j\in \mathcal{L}(n,m,\lambda)$ with $\lambda\in\mathcal{Z}(n,m)$ such that $\Phi(\lambda)=(p,q)$, 
			the Newton polygon of $P(f)$ is equal to the Newton polygon of the set  
			\begin{equation}\label{key}
				E_{\lambda}\ \cup\ \left\{(i,j-1),(i-1,j):\ (i,j)\in T_{\lambda}\cap\supp (f) \right\}.\end{equation}		
			In particular, the Newton polygon of $P(f)$ is equal to the Newton polygon of the generic polar of $y^n-x^m+x^{p}y^{q}+\sum_{(i,j) }a_{i,j}x^iy^j$, where the sum runs on $T_{\lambda}\cap\supp (f)$ and, consequently if $T_{\lambda}=\emptyset$ then the Newton polygon of $P(f)$ is  
			$\mathcal{NP}(E_{\lambda})$.
		\end{teo}
		\begin{proof}
			As a consequence of Lemma \ref{lado-1}, the Newton polygon of the generic polar curve $P(f)$ for any $f=y^n-x^m+x^{p}y^{q}+\sum_{(i,j)\in \mathcal I_{\lambda} }a_{i,j}x^iy^j\in\mathcal{L}(n,m,\lambda)$ is equal to the Newton polygon of the set given in \eqref{eq:points}.
			
			Notice that any point $(i,j)\in \mathcal I_{\lambda}$ is above the line $l_{\sigma,A}$, in particular those points $(i,j)\in \mathcal I_{\lambda}$ with $0\leq j<q$.
			
			Remember that $l_{A,Q}$ is the line determined by the points $A=(p,q)$ and $Q=(m-1,1)$. Since $l_{A,Q}$ is parallel to the line $l_{B,M}$ passing by the points $B=(p,q-1)$ and $M=(m-1,0)$, any point $(i,j)\in \supp (f)$ above $l_{A,Q}$ with $1\leq j\leq q-1$ give us the points $(i,j-1), (i-1,j)\in \Theta_f$ that are above the line $l_{B,M}$, in particular, $(i,j-1),(i-1,j)\in\mathcal{ND}(E_{\lambda})\setminus\mathcal{NP}(E_{\lambda})$. Since
			$$\mathcal{ND}(E_{\lambda})\subseteq \mathcal{ND}(P(f)),$$ the points $(i,j-1)$ and $(i-1,j)$ in $\Theta_f$ and above the line $l_{B,M}$ do not affect the Newton diagram  of $P(f)$.
			
			In this way, for any $f\in \mathcal{L}(n,m,\lambda)$ the Newton polygon of the generic polar $P(f)$ coincides with the Newton polygon of the set $E_{\lambda}$ and the points $(i-1,j)$ and $(i,j-1)$ such that $(i,j)\in \supp(f)$ with $1\leq j\leq q-1$ that are above the line $l_{\sigma,A}$ and below or on the line $l_{A,Q}$, that is, the  Newton polygon of $P(f)$ is equal to the Newton polygon of the set
			\[
			E_\lambda\ \cup\ \left\{(i,j-1),(i-1,j):\ (i,j)\in T_{\lambda}\cap\supp (f) \right\}.
			\]
			
			In particular, if $T_{\lambda}=\emptyset$, then the Newton polygon of $P(f)$ is $\mathcal{NP}(E_{\lambda})$.
		\end{proof}
		
		In \cite[Section 3]{Casas}, Casas-Alvero considers the triangle $\hat{\Lambda}$ with vertices $(0,n),(m-1,1)$ and 
		$(m-\frac{m}{n},1)$. He proves that the Newton polygon of $P(f)$ for any element $f\in K(n,m)$ coincides with the Newton polygon of the generic polar of $y^n-x^m+\sum_{(i,j)\in\hat{\Lambda}}a_{i,j}x^iy^j$.
		
		 For $f\in\mathcal{L}(n,m,\lambda)$ with $\lambda>2m-n$, the Newton polygon of the generic polar curve equals to the Newton polygon of $E_{\lambda}$ (see Lemma \ref{T=0} and Theorem \ref{prop-triang1}). In Theorem \ref{theo} we will describe the topology of the polar in this case. 
		 
		 On the other hand, it follows, by Remark \ref{p+q}, that for any $\lambda<2m-n$ such that $\Phi(\lambda)=(p,q)$ we get $(p,q)\in\hat{\Lambda}$ and $T_{\lambda}\subsetneq\hat{\Lambda}$. In this way, the set $T_{\lambda}$ provides a refined set defining the Newton polygon of the generic polar curve $P(f)$ for any $f\in\mathcal{L}(n,m,\lambda)$ with $\lambda<2m-n$.
		
		Recall that any curve in $\mathcal{L}(n,m,\lambda)$ is analytically equivalent to a curve given by (\ref{equationf}), that is, $	f(x,y)=y^n-x^m+x^py^q+\sum_{(i,j)\in \mathcal{I}_\lambda}a_{ij}x^iy^j$ with $\Phi(\lambda)=(p,q)$. We say that $f$ is {\it generic} in $\mathcal{L}(n,m,\lambda)$ if the coefficients $a_{ij}$ of $f$ with $(i,j)\in\mathcal{I}_{\lambda}$ belong to an open Zariski set.
		
		In the following proposition, we will describe a set of points whose Newton polygon coincides with $\mathcal{NP}(P(f))$ for a generic element $f\in\mathcal{L}(n,m,\lambda)$.
		
		\begin{prop}\label{generic-polar}
			The Newton polygon of the generic polar curve $P(f)$ for a generic element $f\in\mathcal{L}(n,m,\lambda)$ equals the Newton polygon of the set $E_{\lambda}\cup S$ where $S=\emptyset$ if $\Phi(\lambda)=(p,1)$ or
			
			\[	S=\left \{\begin{array}{ll}\left(p+\left \lceil \frac{(q-j-1)m}{n}\right \rceil, j\right): &  0\leq j \leq q-2,\ \ \ \ \mbox{and}\vspace{0.2cm}\\ &  \left\lceil \frac{(q-j)m}{n}\right \rceil\leq \frac{(q-j)(m-p-1)}{q-1}\end{array}\right \} \]
			if $\Phi(\lambda)=(p,q)$ with $q>1$.
		\end{prop}
		\begin{proof}
			By Theorem \ref{prop-triang1}, for any $f\in\mathcal{L}(n,m,\lambda)$ (generic or not) the Newton polygon $\mathcal{NP}(P(f))$ of $P(f)$ coincides with the Newton polygon of the set given in (\ref{key}).
			
			If $q=1$ then it follows, by item 2 of Lemma \ref{T=0}, that $T_{\lambda}=\emptyset$ and consequently, $\mathcal{NP}(P(f))=\mathcal{NP}(E_{\lambda})$ for any $f\in\mathcal{L}(n,m,\lambda)$ generic or not.	
			
			Let us suppose that $q>1$ and consider a generic  $f\in\mathcal{L}(n,m,\lambda)$ given as in (\ref{equationf}) whose coefficients $a_{ij}$ live in the open Zariski set $U:=\{ a_{ij}\neq 0:\ (i,j)\in T_{\lambda}\}$, that is  for a generic  $f\in\mathcal{L}(n,m,\lambda)$, we get $T_{\lambda}\cap \supp(f)=T_{\lambda}$, so the Newton polygon of $P(f)$ coincides with the Newton polygon of the set 
			\begin{equation}\label{key2}
				E_{\lambda}\ \cup\ \left\{(i,j-1),(i-1,j):\ (i,j)\in T_{\lambda} \right\}.\end{equation}	
						
			Notice that the point $(i,j-1)$ and $(i-1,j)$ in (\ref{key2}) arise from  the point $(i,j)\in T_{\lambda}\subset\mathcal{I}_{\lambda}$ when we consider $f_y$ and $f_x$ respectively.
			If $(i,j)\in \mathcal{I}_{\lambda}$ then $i>p+\left ( \frac{(q-j)m}{n}\right )$.
					
			Therefore, $i=p+\left [ \frac{(q-j)m}{n}\right ]$ is the lowest possible value of $i$ such that $(i,j)\in\supp (f)$.
			
			Hence, for $(i,j)\in T_{\lambda}$ we get:
			\begin{itemize}
				\item the minimal value of $i$ such that $(i,j-1)\in \supp (f_y)$ is $p+\left [ \frac{(q-j)m}{n}\right ]$;
				\item the minimal value of $i$ for $(i-1,j)\in supp(f_x)$ is $p+\left [ \frac{(q-j)m}{n}\right ]-1$.
			\end{itemize}		
			Since for every $(i,j)\in\mathcal{I}_{\lambda}$ we get
			\[ p+\left[\frac{(q-j-1)m}{n}\right] \leq p+\left[\frac{(q-j)m}{n}\right]-1,\]
			\noindent and  it follows that
			$i(\rho)=p+\left[\frac{(q-\rho-1)m}{n}\right]$,
			is the minimal value of $i$ such that $(i(\rho),\rho)\in\supp(P(f))$ for $0\leq\rho\leq q-2$, that is, the Newton polygon of $P(f)$ equals to the Newton polygon of the set $E_{\lambda}\ \cup\ \{(i(\rho),\rho):\ 0\leq\rho\leq q-2\}$. However, since the point $\left (p+\left[\frac{(q-\rho-1)m}{n}\right],\rho\right )$ arise the point $\left (p+\left[\frac{(q-j)m}{n}\right],j\right )\in\mathcal{I}_{\lambda}$ for $1\leq j\leq q-1$, it is sufficient to consider the case when the  point $(i(\rho),\rho)$ belongs to $T_{\lambda}$.
			
			Since $q\neq 1$, by (\ref{regiont1}), we have $(i,j)\in T_{\lambda}$ if and only if
			\[ p+\frac{(q-j)m}{n}<i\leq p+\frac{(q-j)(m-p-1)}{q-1},\ 1\leq j\leq q-1\ \ \mbox{and}\ i\neq m-1.\]
			Thus, for each $1\leq j\leq q-1$ we get 
			that $\{i\in\mathbb{N}\ :\ (i,j)\in T_{\lambda}\}$ is the emptyset when $\left\lceil \frac{(q-j)m}{n}\right \rceil> \frac{(q-j)(m-p-1)}{q-1}$; otherwise $\min\{i\in\mathbb{N}\ :\ (i,j)\in T_{\lambda}\}=p+\left\lceil \frac{(q-j)m}{n}\right \rceil$.
			
			Hence, $\mathcal{NP}(P(f))$ coincides with the Newton polygon of the set $E_{\lambda}\ \cup\ \{\left (p+\left[\frac{(q-j-1)m}{n}\right],j\right ):\ 0\leq j\leq q-2\}$ such that $\left\lceil \frac{(q-j)m}{n}\right \rceil\leq \frac{(q-j)(m-p-1)}{q-1}$ and we get the proposition.
		\end{proof}
				 
		By Proposition \ref{prop-triang1}, for any $f\in\mathcal{L}(n,m,\lambda)$  satisfying $\supp (f)\cap T_{\lambda}=\emptyset$ we get $\mathcal{NP}(P(f))=\mathcal{NP}(E_{\lambda})$. In what follows we will determine the topological type of $P(f)$ when its Newton polygon equals the Newton polygon of $E_{\lambda}$. Notice that this corresponds to considering, for instance, the curve of equation 
		\begin{equation}\label{special}
			f=y^n-x^m+x^py^q\in \mathcal L(n,m,\lambda),
		\end{equation}
		\noindent where $\Phi(\lambda)=(p,q)$.
		
		\begin{prop}\label{polar-special}
			Given $f=y^n-x^m+x^py^q \in \mathcal L(n,m,\lambda)$ with $\Phi(\lambda)=(p,q)$ and $q>1$ the topological type of the generic polar $P(f)$ is determined by $n,m$ and $\lambda$ as follows:
			\vspace{0.2cm}
			
			\noindent {\bf 1)} If $\lambda<p+q+m-n$ then $P(f)$ has $d_1:=\gcd(n-q,p)$ branches $\{P_i\}_{i=1}^{d_1}$ with semigroup of values $\langle \alpha_1, \beta_1\rangle$ where $\alpha_1=\frac{n-q}{d_1}$ and $\beta_1=\frac{p}{d_1}$; and $d_2:=\gcd(q-1,m-p-1)$ branches $\{Q_i\}_{i=1}^{d_2}$ with semigroup of values $\langle \alpha_2, \beta_2\rangle$ where $\alpha_2:=\frac{q-1}{d_2}$ and $\beta_2:=\frac{m-p-1}{d_2}$. Moreover,		$\I(P_i,P_j)=\alpha_1\beta_1$, $\I(Q_i,Q_j)=\alpha_2\beta_2$ and $\I(P_i,Q_j)=\min\{\alpha_1\beta_2, \alpha_2\beta_1\}$ for $i\neq j$.\vspace{0.2cm}
			
			\noindent {\bf 2)} If $\lambda\geq p+q+m-n$ then $P(f)$ is the union  of $d:=\gcd(n-1,m-1)$ equisingular branches $P_i$ with semigroup of values $\langle \alpha, \beta\rangle$
			where  $\alpha=\frac{n-1}{d}$ and $\beta=\frac{m-1}{d}$ and such that $\I(P_i,P_j) =\alpha \beta$, for $i\neq j$.
		\end{prop}
		\begin{proof}
			By Proposition \ref{prop-triang1},
			the Newton polygon of the generic polar curve $P(f)$ equals the Newton polygon of the set $E_\lambda$. We distinguish two cases:
			\vspace{0.2cm}
			
			\noindent {\bf 1)}  $\lambda<p+q+m-n$ that is, the point $(p,q-1)$ is below the line $l_{N,M}$.
			
			In this case the Newton polygon of $P(f)$ has two compact edges contained in  $l_{N,B}$ and $l_{B,M}$ respectively.
			
			According to Proposition \ref{Oka}, 
			associated with the compact edge contained in $l_{N,B}$ we get  $d_1:=\gcd(n-q,p)$ branches $\{P_i\}_{i=1}^{d_1}$ with semigroup of values  $\langle \alpha_1, \beta_1\rangle$ where $\alpha_1=\frac{n-q}{d_1}$ and $\beta_1=\frac{p}{d_1}$.
			Associated with the compact edge contained in $l_{B,M}$ we get  $d_2=\gcd(q-1,m-p-1)$ branches $\{Q_i\}_{i=1}^{d_2}$ with semigroup of values  $\langle \alpha_2, \beta_2\rangle$ where $\alpha_2=\frac{q-1}{d_2}$ and $\beta_2=\frac{m-p-1}{d_2}$. Moreover, for $i\neq j$, 
			$\I(P_i,P_j)=\alpha_1\beta_1$, $\I(Q_i,Q_j)=\alpha_2\beta_2$ and $\I(P_i,Q_j)=\min\{\alpha_1\beta_2, \alpha_2\beta_1\}$.
			\vspace{0.2cm}
			
			\noindent {\bf 2)} $\lambda\geq p+q+m-n$, that is  the point $(p,q-1)$ is on or above the line $l_{N,M}$.	
			
			In this case the Newton polygon of the generic polar curve $P(f)$ has only one compact edge $L$ contained in the line $l_{N,M}$. 
			The univariate polynomial associated with $L$ is $(P(f))_L(1,t)=bnt^{n-1}-am$ (when $(p,q-1)$ is above $l_{N,M}$) or $(P(f))_L(1,t)=bnt^{n-1}+bqt^{q-1}-am$ (when $(p,q-1)$ is on $l_{N,M}$). Since we are considering the generic polar of $f$ in both cases there exists an open set $U\subset\mathbb{P}^1$ such that for any $(a:b)\in U$ the polar $P(f)$ is Newton non-degenerate. In particular, by Proposition \ref{Oka}, the generic polar $P(f)$ is the union  of $d=\gcd(n-1,m-1)$ equisingular branches $\{P_i\}_{i=1}^d$ with $\Gamma_{P_i}=\langle \alpha, \beta \rangle$, for any $i$, where $\alpha=\frac{n-1}{d}, \beta=\frac{m-1}{d}$ and $\I(P_i,P_j) =\alpha \beta$ for $i\neq j$. 
		\end{proof}
		
		As a consequence of Lemma \ref{T=0} and Proposition \ref{polar-special} we have 
		
		\begin{corolario}
			Let $f\in\mathcal{L}(n,m,\lambda)$ where $\lambda\in\mathcal{Z}(n,m)\cup\{\infty\}$ and with $\Phi(\lambda)=(p,q)$ if $\lambda\neq\infty$.
			\begin{enumerate}
				\item If $\lambda>2m-n$ (admitting $\lambda=\infty$) then the generic polar curve $P(f)$ is  equisingular to the monomial curve $y^{n-1}-x^{m-1}=0$. 
				Moreover, if $d:=\gcd(n-1,m-1)$ then $P(f)$ have $d$ irreducible equisingular  components $\{P_i\}_i^{d}$ such that $\Gamma_{P_i}=\left \langle \frac{n-1}{d},\frac{m-1}{d} \right \rangle$ for any $i\in \{1,\ldots,d\}$ and $\I(P_i,P_j)=\frac{(n-1)(m-1)}{d^2}$ for any $i\neq j$.
				\item If $q=1$ then the generic polar curve 
				$P(f)$ is  equisingular to the monomial curve $y^{n-1}-x^p=0$. Moreover, if $d:=\gcd(n-1,p)$ then $P(f)$ have $d$ irreducible equisingular components $\{P_i\}_i^{d}$ such that $\Gamma_{P_i}=\left \langle \frac{n-1}{d},\frac{p}{d}\right \rangle$ for any $i\in \{1,\ldots,d\}$ and $\I(P_i,P_j)=\frac{(n-1)p}{d^2}$ for any $i\neq j$.
				\item For any $f\in K(3,m)$ the generic polar curve $P(f)$ has the topological type described in item 1) if $\lambda=\infty$ or as described in item 2) if $\lambda\neq\infty$.
				\item For any $f\in K(n,n+1)$ the generic polar curve $P(f)$ has the topological type  described in item 1).
			\end{enumerate}
		\end{corolario}
		\begin{proof}
			First of all, notice that in all cases we get, by Lemma \ref{T=0}, $T_{\lambda}=\emptyset$. In this way, by Theorem \ref{prop-triang1}, it follows that $\mathcal{NP}(P(f))=\mathcal{NP}(E_{\lambda})$ where $E_{\lambda}=\{(0,n-1),(m-1,0),(p,q-1)\}$. So, the topological type of $P(f)$ is the same of the topological type of the generic polar curve of $y^n-x^m+x^{p}y^{q}$ as given in (\ref{special}).
			
			If $\lambda>2m-n$ then $(p,q-1)$ is above the line $l_{N,M}$ and the result follows by item 2 of Proposition \ref{polar-special}.
			
			If $q=1$ then $\mathcal{NP}(P(f))=\mathcal{NP}(E_{\lambda})$ is equal to the Newton polygon of the set $\{(0,n-1), (p,0) \}$ and the result follows from Proposition \ref{Oka}.
			
			If $f\in K(3,m)$ and $\lambda_f=\infty$ the item 1) gives the result. On the other hand, if $\lambda_f\neq\infty$ then we get $q=1$ for any $\lambda\in\mathcal{Z}(3,m)$ and we are in the hypothesis of item 2). 
			
			If $f\in K(n,n+1)$ then, as we have verified in the proof of Lemma \ref{T=0}, any $f\in K(n,n+1)$ satisfies $\lambda_f>2m-n$ and item 1) give us the result.
		\end{proof}
		
		In the following theorem, we summarize all the previous results concerning the topological type of the generic polar $P(f)$ for $f\in\mathcal{L}(n,m,\lambda)$.
		
		\begin{teo}\label{theo}
			Let $f=y^n-x^m+x^{p}y^{q}+\sum_{(i,j)\in \mathcal I_{\lambda} }a_{i,j}x^iy^j\in \mathcal{L}(n,m,\lambda)$ for $\lambda\in\mathcal{Z}(n,m)$ such that $\Phi(\lambda)=(p,q)$. Then
			\begin{enumerate}
				\item If $q=1$, that is, $\lambda=2m-(m-p)n$ then $P(f)$ is topologically equivalent to $y^{n-1}-x^p$.
				\item If $\lambda >2m-n$ then $P(f)$ is topologically equivalent to $y^{n-1}-x^{m-1}$.
				\item If $\lambda < 2m-n$ and $q>1$ then the Newton polygon of $P(f)$ is equal to the Newton polygon of the generic polar of \[y^n-x^m+x^{p}y^{q}+\sum_{(i,j)\in T_{\lambda}\cap \,{\small {\supp} (f)} }a_{i,j}x^iy^j.\]
    \noindent Moreover, if $T_{\lambda}=\emptyset$  then  the topological type of $P(f)$ is given by Proposition \ref{polar-special}.
			\end{enumerate}
		\end{teo}
		
		\begin{ejemplo}\label{example}
			
			In the sequel we apply our results to describe the topological type for plane branches in $K(5,12)$, that is, $n=5$ and $m=12$. 
			
			\noindent The set of possible finite Zariski invariants in this equisingularity class  is (see the right side of Figure \ref{fig:triangle2en1})
			\[
			\mathcal{Z}(5,12)=\{13,14,16,18,21,23,26,28,33,38\}.
			\] 
			
			Since $\Phi(14)=(10,1)$, by item 1 of Theorem \ref{theo}, for any $f\in\mathcal{L}(5,12,14)$ the generic polar $P(f)$ is topologically equivalent to $y^4-x^{10}=(y^2-x^5)(y^2+x^5)$, that is, $P(f)$ has two equisingular branches with semigroup $\langle 2,5\rangle$ and the intersection multiplicity of them is $10$.
			
			By item 2 of Theorem \ref{theo}, for any $f\in\mathcal{L}(5,12,\lambda)$ with $\lambda>19=2m-n$, that is $\lambda\in\{21,23,26,28,33,38\}$, the generic polar $P(f)$ is topologically equivalent to $y^4-x^{11}$. So, $P(f)$ is irreducible with semigroup $\langle 4,11\rangle$.
			
			After the description of $T_{\lambda}$ in (\ref{regiont1}) we get $T_{16}=\emptyset$.  Since $\Phi(16)=(8,2)$ and $\lambda=16<17=8+2-5+12=p+q-n+m$, according to item 1 of Proposition \ref{polar-special}, the generic polar for any $f\in\mathcal{L}(5,12,16)$ has a branch $P_1$ with semigroup $\langle 3,8\rangle$ and a smooth branch $Q_1$ with $\textup{I}(P_1,Q_1)=8$.
			
			For $\lambda=18$ we also get $T_{18}=\emptyset$. Since $\Phi(18)=(6,3)$ and $\lambda=18>16=6+3-5+12=p+q-n+m$ then, by item 2 of Proposition \ref{polar-special} the generic polar for any $f\in\mathcal{L}(5,12,18)$ is irreducible with semigroup $\langle 4,11\rangle$.
			
			Consider now $\lambda=13$. We get $\Phi(13)=(5,3)$ and $T_{13}=\{(10,1),(8,2)\}$ (see Figure \ref{regiont1}). So, by item 3 of Theorem \ref{theo}, for any $f\in\mathcal{L}(5,12,13)$ the Newton polygon of $P(f)$ is equal to the Newton polygon of the generic polar of $y^5-x^{12}+x^{5}y^{3}+a_{10,1}x^{10}y+a_{8,2}x^8y^2$. We have the following possibilities:
			\begin{itemize}
				\item If $a_{10,1}=a_{8,2}=0$, then $T_{13}\cap\supp (f)=\emptyset$. Since $\lambda=13<15=p+q-n+m$, the item 1 of Proposition \ref{polar-special} allows us to conclude that the generic polar $P(f)$ has one branch $P_1$ with semigroup $\langle 2,5\rangle$,  two  branches $Q_1$ and $Q_2$ with semigroup $\langle 1,3\rangle=\mathbb{N}$ with $\textup{I}(Q_1,Q_2)=3$ and $\textup{I}(P_1,Q_i)=5$ for $i=1,2$.
				\item If $a_{10,1}\neq 0$ then $\mathcal{ND}(P(f))=\mathcal{ND}(\{(0,4),(0,10),(5,2)\})$. For $a_{10,1}\neq 9/20$ the generic polar has two equisingular branches $Q_1$ and $Q_2$ with semigroup $\langle 2,5\rangle$ and $\textup{I}(Q_1,Q_2)=10$. For $a_{10,1}=9/20$ the generic polar is degenerate and it corresponds to a branch with semigroup $\langle 4,10,21\rangle$. In both cases, we get the topological type by the computation of the first terms of the Puiseux parametrizations of $P(f)$.
				\item If $a_{10,1}=0$ and $a_{8,2}\neq 0$ then the Newton diagram of $P(f)$ is \[\mathcal{ND}(\{(0,4),(5,2),(8,1),(11,0\})\] and computing the Puiseux parametrizations of $P(f)$ we conclude that the topological type of 
				$P(f)$ is equal to the case $a_{10,1}=a_{8,2}=0$.
			\end{itemize}  
		\end{ejemplo}
		
		In \cite{Hefez-Hernandes-Iglesias2017} is presented the description of the all possible topological type of the polar generic curve for any plane curve in $K(5,12)$. The method employed consists in, fixing an analytical invariant denoted by $\Lambda$ and for each associated normal form of $\Lambda$ given by a Puiseux parametrization, considering the equation of each normal formal to analyze the possible topological type of the generic polar curve. That result is in line with our study when considering  $f\in\mathcal{L}(n,m,\lambda)$, but differ for the description of particular value of parameter since the parameters in the parametrization are related by are not the same of parameters in the expression $f\in\mathbb{C}\{x,y\}$. Notice that for $f\in\mathcal{L}(5,12,\lambda)$, except for the case $\lambda=13$, we have the description of the topological type of $P(f)$ directly by $n, m$ and $\lambda$, that is can be obtained independently how the curve is presented, by Puiseux parametrization or by an equation.

		\vspace{0.5cm}
		
		\noindent {Evelia Rosa \sc  Garc\'{\i}a Barroso}\\
		Departamento de Matem\'aticas, Estad\'{\i}stica e I.O.\\
		IMAULL. Universidad de La Laguna.\\
		Apartado de Correos 456.
		38200 La Laguna, Tenerife, Espa\~na 
		\\ORCID ID: 0000-0001-7575-2619
		
		\noindent {ergarcia@ull.es}
		\vspace{0.3cm}
		
		\noindent {Marcelo Escudeiro \sc  Hernandes}\\
		Departamento de Matem\'atica. \\
		Universidade Estadual de Maring\'a. \\
		Avenida Colombo 5790. Maring\'a-PR 87020-900.
		Brazil. 
		\\ORCID ID: 0000-0003-0441-8503
		
		\noindent {mehernandes@uem.br}
		\vspace{0.3cm}
		
		\noindent {Mauro Fernando \sc Hern\'andez Iglesias}\\
		Pontificia Universidad Cat\'olica del Per\'u. \\
		Av. Universitaria 1801, San Miguel 15088, Per\'u. 
		\\ ORCID ID: 0000-0003-0026-157X
		
		\noindent {mhernandezi@pucp.pe}

	\end{document}